\documentclass[12pt]{article}

\usepackage{graphicx}
\usepackage{amsmath}
\usepackage{amsthm}
\usepackage{amsfonts}
\usepackage{amssymb}
\usepackage{diagrams}
\usepackage{enumitem}

\usepackage[utf8]{inputenc}
\usepackage[english]{babel}
\usepackage[backend = biber,
    style = alphabetic,
    citestyle = alphabetic
]{biblatex}
\usepackage{hyperref}
\hypersetup{
 colorlinks,
 linkcolor=blue
 }

\theoremstyle{definition}

\newtheorem{Thrm}{Theorem}[section]
\newtheorem{Cor}[Thrm]{Corollary}
\newtheorem{Fact}[Thrm]{Fact}
\newtheorem{Lem}[Thrm]{Lemma}
\newtheorem{Prop}[Thrm]{Proposition}

\newtheorem{Def}[Thrm]{Definition}

\newtheorem{Not}[Thrm]{Notation}
\newtheorem{Rem}[Thrm]{Remark}

\newcommand{\PS}{\mathbb{P}}

\newcommand{\Z}{\mathbb{Z}}

\newcommand{\Fc}{\mathcal{F}}

\newcommand{\Lc}{\mathcal{L}}
\newcommand{\Mc}{\mathcal{M}}
\newcommand{\Nc}{\mathcal{N}}
\newcommand{\Oc}{\mathcal{O}}
\newcommand{\Pc}{\mathcal{P}}

\newcommand{\Tc}{\mathcal{T}}
\newcommand{\Qc}{\mathcal{Q}}

\newcommand{\ch}{\mathrm{ch}\,}
\newcommand{\chn}{\mathrm{ch}}
\newcommand{\cn}{\mathrm{c}}
\newcommand{\Coh}{\mathrm{Coh}\,}

\newcommand{\Chown}{\mathrm{A}}

\newcommand{\Ext}{\mathrm{Ext}}
\newcommand{\fdeg}{\mathrm{fdeg}\,}
\newcommand{\FM}{\mathrm{FM}}
\newcommand{\FMXY}{\mathrm{FM}_{X \to Y}}
\newcommand{\FMYX}{\mathrm{FM}_{Y \to X}}

\newcommand{\Lab}{\mathrm{L}^0}
\newcommand{\NS}{\mathrm{NS}\,}

\newcommand{\Pic}{\mathrm{Pic}\,}
\newcommand{\Picnot}{\mathrm{Pic}^\circ\,}

\newcommand{\Rab}{\mathrm{R}^0}

\newcommand{\rk}{\mathrm{rk}\,}

\newcommand{\SLZ}{\mathrm{SL}_2\Z}

\newcommand{\td}{\mathrm{td}}

\newcommand{\Chow}{\mathrm{A}^\bullet}
\newcommand{\K}{\mathrm{K}_0\,}
\newcommand{\Kany}{\mathrm{K}_0^*\,}
\newcommand{\Kor}{\mathrm{K}_0^{\mathrm{or}}\,}

\newcommand{\Kt}{\mathrm{K}_0^{\mathrm{t}}\,}

\newcommand{\vor}{^{\mathrm{or}}}

\newcommand{\vt}{^{\mathrm{t}}}

\newcommand{\Db}{\mathrm{D^b}}

\newcommand{\Hom}[2]{\mathrm{Hom}_{#1} \left( #2 \right)}
\newcommand{\Hrm}{\mathrm{H}}
\newcommand{\RGamma}{\mathrm{R}\Gamma}
\newcommand{\RHom}[2]{\mathrm{RHom}_{#1} \left( #2 \right)}

\newcommand{\deq}{\overset{\text{def}}{=}}

\newcommand{\fprod}[1]{\underset{#1}{\times}}

\begin{document}

\title{Strange Duality for elliptic surfaces}
\author{Svetlana Makarova} 
\date{}

\maketitle

\begin{abstract}

The main result of the present paper is the proof of the Strange Duality for elliptic surfaces --
a duality between global sections of determinantal line bundles on moduli spaces of stable sheaves on a fixed elliptic surface.
For this, we employ the ``Marian-Oprea trick'': using Bridgeland's birational isomorphisms, we reduce the problem from a pair of general moduli spaces to a pair of Hilbert schemes.
The latter case is a theorem by Marian-Oprea.
\end{abstract}

\tableofcontents

\section{Introduction}

The work on the present paper started with an attempt to strengthen the results on the Strange Duality on surfaces, and is largely motivated by the approach of Marian and Oprea \cite{MO}. 
The Strange Duality is a conjectural duality between global sections of two natural line bundles on moduli spaces of stable sheaves. It originated as a representation theoretic observation about pairs of affine Lie algebras, and then was reformulated geometrically.
In our paper, we develop the geometric approach in the spirit of Marian and Oprea. They proved the Strange Duality conjecture for Hilbert sche\-mes of points on surfaces and for moduli of sheaves on abelian surfaces \cite{MO08}, and in later work for elliptic K3 surfaces with a section \cite{MO}. The latter used birational isomorphisms of moduli spaces of stable sheaves with Hilbert schemes of points on the same K3 surface, following Bridgeland \cite{Br}, to reduce the question to the known case of Hilbert schemes.

\paragraph{Outline of the paper.}

We start with \S\ref{Sec:FM}:
in the first part, we summarize relevant results of Bridgeland \cite{Br}; and in the second part, we obtain a mild generalization of those that we use in the main body of the paper. For moduli of stable sheaves of rank at least three, we automatically get that the singular locus lives in codimension two, because the rational map is defined over the locus of vector bundles.
We next show in \S\ref{Sec:SD} that the Strange Duality holds for a range of elliptic surfaces (not necessarily K3) and K-theory classes in Theorem \ref{Thrm:SD_elliptic}, which extends the pool of evidence of the Strange Duality conjecture to elliptic surfaces.
Finally, in \S\ref{Sec:matrices}, we present new universal sheaves on the fiber square of an elliptic K3 surface with a section, in our first attempt to relativize Atiyah's construction of stable vector bundles on elliptic curves \cite{At}.

\paragraph{Conventions.}
We work over an algebraically closed field of characteristic zero.
We write $(-)^\vee$ for the derived dual of a sheaf and $-\otimes -$ for the derived tensor product.
Given a morphism of schemes $f : X \to Y$, we denote by $f_*$ and $f^*$ the derived functors of pushforward and pullback, respectively. When we want to work with the classical functors instead of derived, we write $\Lab f^*$ for the nonderived pullback and $\Rab f_*$ for nonderived pushforward.
Note however that we distinguish between $\mathrm{Hom}$ and $\mathrm{RHom}$ (because $\mathrm{Hom}$ makes sense in the derived category on its own).


We use the equality sign to denote natural isomorphisms (as well as equality).


For the moduli theory of sheaves, we generally need a way to fix a numerical characteristic of the sheaves in question in order to obtain any finiteness results.
So, for a variety $X$, we use zeroth algebraic K-theory $\K X$ and zeroth topological K-theory $\Kt X$. The latter is rationally isomorphic to the ring of even cohomology groups. 
We can also define oriented topological K-theory $\Kor X$ by fixing the determinant of a topological K-theory vector. 
We will call a vector $v$ in any K-theory $\Kany X$ a fixed
\emph{numerical characteristic}, or
\emph{K-theory class}. When we need to be specific, we will add adjectives algebraic, topological or oriented topological to refer to the corresponding variants of K-theory.

Let $\Chow X$ denote the Chow ring of a smooth projective variety $X$.
It is well-known that there is a function called Chern character $\ch : \Db X \to \Chow X$, from objects of the derived category to the Chow ring, that factors as a ring homomorphism through the Grothendieck group: $\ch : \K X \to \Chow X$. Note that the Euler pairing descends to each of the K-groups by taking representative complexes $E$ and $F$ and computing Euler characteristic of their derived tensor product:
$$\chi(E \otimes F) \deq \chi \left( \RGamma(E \otimes F) \right)
\mbox{.}$$
Moreover, this pairing is intertwined with the morphism $\ch : \K X \to \Chow X$ via Hirzebruch-Riemann-Roch.

We don't use Chern classes a lot, and instead we prefer to write a K-theory vector $v$ in terms of components of its Chern character: $\chn_0 v = \rk v$, $\chn_1 v = \cn_1 v$, $\chn_2 v = \frac{1}{2} (\cn_1 v)^2 - \cn_2 v$, etc.

\paragraph{Acknowledgements.}
First and foremost, I would like to thank my advisor Davesh Maulik for suggesting the topic and insightful discussions throughout the work.
I have benefitted from discussions related to this work with
Alina Marian (my deepest appreciation for helping me escape deadends and commenting on the paper draft),
Dragos Oprea (for the comments on the paper draft),
Valery Alexeev,
Arend Bayer,
Dori Bejleri,
Tom Bridgeland,
Elden Elmanto, 
Nikon Kurnosov,
Emanuele Macr\`i,
Eyal Markman,
Evgeny Shinder,
Kota Yoshioka,
Xiaolei Zhao.
\section{Fourier-Mukai transforms and birational isomorphisms}

\label{Sec:FM}

In this section, we will formulate the main ingredients of Bridgeland's work \cite{Br} and later use some of them to obtain mild generalizations.
\S\ref{Subsec:FM_stage} does not contain original results and instead focuses on giving a sufficient review. The reader who wants more details and references is advised to turn their attention to the original work \cite{Br}, which is a pleasant read on its own.
In \S\ref{Subsec:FM_general}, we formulate our mild generalization of Bridgeland's technique, and the results are used to justify conclusions of \S\ref{Sec:SD}.

\subsection{Setting the stage}
\label{Subsec:FM_stage}

\paragraph{General properties of Fourier-Mukai transforms.}

Take smooth pro\-per varieties $X$ and $Y$ and an object $\Pc \in \Db (X \times Y)$. Let $\pi_X$ and $\pi_Y$ denote the two projections from $X \times Y$ to $X$ and $Y$, respectively.

\begin{Def}
    Define the \emph{Fourier-Mukai functor}
    $\FMYX^\Pc : \Db Y \to \Db X$
    \emph{with kernel} $\Pc$ by the following formula:
    $$  \FMYX^\Pc (-) \deq \pi_{X*} \left(
        \Pc \otimes \pi_Y^* (-) \right)
    \mbox{.} $$
\end{Def}

\begin{Def}
    Let $\Phi : \Db Y \to \Db X$ be an exact functor.
    We say that a sheaf $E \in \Coh Y$ is \emph{$\Phi$-WIT$_i$}
    if $\Phi(E)[i]$ is a sheaf (instead of a complex).
\end{Def}

\begin{Fact}
    There exists a group homomorphism
    $\ch (\Pc) = \ch (\FMYX^\Pc) : \Chow Y \to \Chow X$
    that makes the following square commute:
    \begin{diagram}
    \Db Y       & \rTo^{\FMYX^\Pc}  & \Db X     \\
    \dTo^{\ch}  &                   & \dTo_{\ch}\\
    \Chow Y     & \rTo^{\ch(\Pc)}   & \Chow X   
    \end{diagram}
    It is given by the formula:
    $$  \ch (\Pc) (y) = \pi_{X*} \big( \ch(\Pc) . \pi_Y^* (\td_Y.y) \big)
    \mbox{.} $$
\end{Fact}

\begin{proof}
    The proof uses Grothendieck-Riemann-Roch theorem and properness of the varieties, so that $\pi_{X*} = \pi_{X!}$. 
\end{proof}

\begin{Def}
    A $Y$-flat sheaf $\Pc$ on $X \times Y$ is said to be \emph{strongly simple} over $Y$ if $\Pc_y \deq \Pc_{|X \times y}$ is simple for all $y \in Y$ and for any pair of distinct points $y, y' \in Y$ one has $\RHom{X}{\Pc_y , \Pc_{y'}} = 0$.
\end{Def}

\begin{Thrm}[Mukai; Bondal-Orlov]
    Let $\Pc$ be a $Y$-flat sheaf on $X \times Y$. Then $\FMYX^\Pc$ is fully faithful iff $\Pc$ is strongly simple over $Y$. If $\Pc$ is flat over both factors, then $\FMYX^\Pc$ is an equivalence iff $\Pc$ is strongly simple over both factors.
\end{Thrm}

\paragraph{Fourier-Mukai transforms for elliptic surfaces.}
Let $\pi : X \to C$ be an elliptic surface. Given a sheaf $E$ on $X$, we write its Chern class as a triple $(\rk E , \cn_1 E , \chn_2 E) \in \Kt X$, and for the rest of this section, the vector will always be picked in topological K-theory.
All fibers of $\pi$ are (possibly singular) elliptic surfaces that are algebraically equivalent to each other.
We denote the fiber class of $\pi$ by $f$.

\begin{Def}
    For a complex $E \in \Db X$, let $\fdeg E \deq \deg \, (\cn_1 E.f)$ denote the \emph{fiber degree} of $E$. Let $\lambda \deq \lambda_X$ denote the smallest positive fiber degree of a sheaf on $X$, or equivalently, the smallest possible degree of a multisection. We say that a sheaf $E \in \Coh X$ is a \emph{fiber sheaf} if $\rk E = \fdeg E = 0$, or equivalently, if $E$ is supported on finitely many fibers of $\pi$. 
\end{Def}

\begin{Thrm}[Bridgeland \cite{Br}, Proposition 4.2, Section 5]
\label{Thrm:Br_birat}
    Let $\pi : X \to C$ be an elliptic surface. Take two coprime integers $a>0$ and $b$. Then there exists a fine moduli scheme $Y \deq J_X(a,b)$ parametrizing stable fiber sheaves on $X$ with an open subset of points representing a rank $a$, degree $b$ vector bundle on a nonsingular fiber of $\pi$. Moreover, the scheme $Y$ enjoys the following properties:
    \begin{enumerate}[label=(\roman*)]
        \item it is endowed with a natural morphism $\pi' : Y \to C$ which makes it into a smooth elliptic surface;
        \item there exists a sheaf $\Pc$ on $X \times Y$,
        flat and strongly simple over both factors,
        such that for every point $y\in Y$, the sheaf $\Pc_y$
        corresponding to $y$ 
        has Chern character
        $(0, af, c)$;
        \item any such sheaf $\Pc$
        defines an equivalence $\Phi \deq \FMYX^\Pc : \Db Y \to \Db X$;
        \item denoting by $\Qc$ the object $(\Pc^\vee \otimes \pi^* \omega_X)[1]$, we have that $\Qc$ is a sheaf on $X \times Y$ that defines an inverse up to a shift equivalence $\Psi \deq \FMXY^\Qc$, that is $\Phi \Psi = [-1]$ and $\Psi \Phi = [-1]$;
        \item there exists an integer $c$ such that $X\cong J_Y(a,c)$;
        \item the functor $\Phi$ satisfies the following matrix equality:
        $$  \begin{pmatrix} \rk (\Phi E) \\ \fdeg (\Phi E)
        \end{pmatrix} =
        \begin{pmatrix} c&a \\ d&b
        \end{pmatrix}
        \begin{pmatrix} \rk E \\ \fdeg E
        \end{pmatrix}
        \mbox{,} $$
        for all objects $E \in \Db Y$ and some $d \in \lambda \Z$ depending on $\Pc$ that makes the matrix belong to $\SLZ$.
    \end{enumerate}
\end{Thrm}

\begin{Rem}
    One can choose $\Pc$ so that $\Pc_x$ is stable for all $x \in X$, but it is important to note that it is in general impossible to choose a tautological sheaf $\Pc$ on $X \times Y$ that would parametrize stable sheaves on both $X$ and $Y$ simultaneously.
\end{Rem}

\paragraph{Application to moduli of stable sheaves.}

Pick a Chern class $v = (r,\Lambda,k)$ in $\Kt X$ such that $r > 1$ and is coprime to the fiber degree $d \deq \Lambda.f$ of $\Lambda$. In this case, Friedman observed that there exist polarizations of $X$ with respect to which a torsion-free sheaf $E$ is stable iff its restriction to all but finitely many fibers is stable, and in this case the notions of Gieseker stability, slope stability and semistability and fiberwise stability all coincide. These polarizations are called \emph{suitable}. For a detailed explanation of these, see Chapter 6 in Friedman's book \cite{Fr}.

For such polarizations, consider the fine moduli space $\Mc \deq \Mc_v(X)$ of stable torsion-free sheaves on $X$.

Let $a$, $b$ be the unique pair of integers satisfying $br-ad = 1$ and $0<a<r$. And let $\pi' : Y \to C$ be the elliptic surface $J_X(a,b)$.

Recall that $Y^{[t]}$ is a standard notation for the Hilbert scheme of $t$ points and consider the moduli space $\Nc \deq \Mc_{(1,0,t)}(Y) = \Picnot Y \times Y^{[t]}$ such that $\dim \Mc = \dim \Nc$.

Recall also that if we have a Fourier-Mukai functor $\Phi : \Db X \to \Db Y$ such that all sheaves in some moduli space $\Mc$ over $X$ are $\Phi$-WIT$_0$, that is are sent to sheaves, then we can say that $\Phi$ induces a map $\Mc \to \Phi(\Mc)$ of sets of sheaves.

\begin{Thrm}[Bridgeland \cite{Br}, Theorem 1.1, Section 7.1]
\label{Br_smooth}
    The space $\Mc = \Mc_v(X)$ is smooth, and in fact the morphism $\det : \Mc \to \Mc_{(1,\Lambda,\frac{\Lambda^2}{2})}(X)$ is smooth.
    The spaces $\Mc$ and $\Nc = \Mc_{(1,0,t)}(Y)$ are birational, in particular, $\Mc$ is irreducible. Moreover, the birational equivalence is established by a Fourier-Mukai functor with kernel $\Pc \in \Coh (X \times Y)$ that is a tautological sheaf parametrizing stable fiber sheaves on $Y$.
\end{Thrm}

\subsection{The generalized statement}
\label{Subsec:FM_general}

Let $\pi : X \to C$ be an elliptic surface and fix a (quasi)polarization $H$ of $X$ such that stability with respect to $H$ is equivalent to stability on fibers.

Fix a triple $v = (r, \Lambda, k)$ in $\Kt X \cong \Z \times \NS X \times \Z$ or in $\Kor X \cong \Z \times \Chown^1 X \times \Z$ in such a way that semistable sheaves of class $v$ are stable. Let $d \deq \Lambda.f$ denote the fiber degree of $\Lambda$, then coprimality of $r$ and $d$ ensures the above requirement.

Take the fine moduli space $\Mc = \Mc_v(X)$ of stable torsion-free sheaves on $X$ of class $v$. We want to state results for either of the K-theories, so we will not make a choice here explicitly. Note that by Bridgeland's Theorem \ref{Br_smooth},
the scheme $\Mc_v$ is smooth for both topological vector $v\vt$ and oriented topological vector $v\vor$.

Fix two matrices $\varphi$ and $\psi$ from $\SLZ$:
$$  \varphi = \begin{pmatrix} c&a \\ e&b \end{pmatrix}
\mbox{, }
    \psi = \begin{pmatrix} -b&a \\ e&-c \end{pmatrix}
\mbox{,}$$
where $e$ is a multiple of $\lambda$ and $a > 0$. They are inverse to each other up to a sign. Set $Y = J_X(a,b)$.

By Theorem \ref{Thrm:Br_birat} (also by \cite{Br}), we can choose sheaves $\Pc$ and $\Qc$ on the product $X \times Y$ which define Fourier-Mukai transforms $\Phi: \Db Y \to \Db X$ and $\Psi: \Db X \to \Db Y$ with matrices $\varphi$ and $\psi$, respectively. We assume that $\Pc$ parametrizes stable torsion-free sheaves on $Y$. So the fibers $\Qc_x$ are also stable fiber sheaves on $Y$, and $\ch \Qc = (0, af, -c)$.


Let $w = \Psi v [1] = - (\ch \Psi) (\ch v)$ and $\Nc = \Mc_w(Y)$.

\begin{Rem}
\label{Rem:numerology}
    In the following sequence of lemmas, we will often have to use the numerical assumption $\frac{b}{a} > \frac{d}{r}$. Geometrically, this means that we want the slope of a stable sheaf $\Pc_y$ on a fiber of $X \to C$ to be greater than the slope of $v$ with respect to $f$; or that the rank of $w$ is positive: $\rk w = br - ad > 0$.
\end{Rem}

\begin{Lem}
    If $\frac{b}{a} > \frac{d}{r}$, then every stable sheaf on $X$ from the moduli space $\Mc = \Mc_{(r, \Lambda, k)}(X)$ is $\Psi$-WIT$_1$.
\end{Lem}

\begin{proof}
    Take $E \in \Mc$. By Bridgeland's Lemma 6.1 \cite{Br}, there exists a short exact sequence $0 \to A \to E \to B \to 0$, where $A$ is $\Psi$-WIT$_0$ and $B$ is $\Psi$-WIT$_1$.

    If $E$ is not WIT$_1$, then $A \neq 0$ and by Bridgeland's Lemma 6.2 \cite{Br},
    $\mu A \geq \frac{b}{a}$.
    But then the inequality
    $\frac{b}{a} > \frac{d}{r} = \mu E$
    contradicts stability of $E$.
\end{proof}

\paragraph{Open loci.} Define the following two subsets:
\begin{equation*}
\begin{split}
    \Mc^\circ &= \left\{ E \in \Mc \mid
    \Psi E [1]
    \mbox{ is torsion-free} \right\}
    \mbox{,} \\
    \Nc^\circ &= \left\{ F \in \Nc \mid F \mbox{ is } \Phi \mbox{-WIT}_0 \right\}
    \mbox{.}
\end{split}
\end{equation*}

\begin{Lem}
    For every $F \in \Nc^\circ$, its transform $\hat F = \Phi F$ is torsion-free and stable, so $\hat F \in \Mc^\circ$.
\end{Lem}

\begin{proof}
    We will adapt the proof of Bridgeland's Lemma 7.2 \cite{Br}.
    
    Observe that $\hat F$ is stable as long as it is torsion-free, by Bridgeland's argument. So we need to prove that $\hat F$ is torsion-free. Assume that $T \subset \hat F$ is a torsion subsheaf, then we get a long exact sequence (\cite{Br}, Equation (6.1.1)) by applying $\Psi$ to the short exact sequence $0 \to T \to \hat F \to \hat F / T \to 0$:
    \begin{equation*}
    \begin{split}
        0 \to \Psi^0 T \to 0 \to 
        \Psi^0 (\hat F / T) \to
        \Psi^1 T \to F \to 
        \Psi^1 (\hat F / T) \to 0
    \mbox{.}
    \end{split}
    \end{equation*}
    It follows that $T$ is a $\Psi$-WIT$_1$ torsion sheaf, so by Bridgeland's Lemma 6.3 \cite{Br} it is a fiber sheaf, and therefore $\Psi^1 T$ also is a fiber sheaf. By assumption, the sheaf $F$ is torsion-free, so from the long exact sequence we get an isomorphism
    $\Psi^0 (\hat F / T) \cong \Psi^1 T$. But the left hand side is a $\Phi$-WIT$_1$ sheaf, while the right hand side is $\Phi$-WIT$_0$, and this can only happen when $\Psi^0 (\hat F / T) \cong \Psi^1 T = 0$, so $T=0$ and $\hat F$ is torsion-free.
\end{proof}

\begin{Lem}
\label{Lem:vb_in_N}
    If $\frac{b}{a} > \frac{d}{r}$, then every vector bundle $F \in \Nc$ is $\Phi$-WIT$_0$, that is lies in the locus $\Nc^\circ$.
\end{Lem}

\begin{proof}
    By Bridgeland's Lemma 6.5 \cite{Br},
    a sheaf $F$ on $Y$ is $\Phi$-WIT$_0$ iff for every $x \in X$, we have the vanishing
    $\Hom{Y}{F,\Qc_x} = 0$.
    
    Recall that $\Qc_x$ is a stable torsion-free sheaf of slope $\mu \Qc_x = -\frac{c}{a}$.
    
    Restriction of $F$ to any fiber is torsion free (because $F$ is a vector bundle) and of slope $\frac{\fdeg F}{\rk F}$, where:
    $$  \begin{pmatrix} \rk F \\ \fdeg F \end{pmatrix}
    = - \psi \begin{pmatrix} \rk v \\ \fdeg v \end{pmatrix} =
    \begin{pmatrix} b&-a \\ -e&r \end{pmatrix}
    \begin{pmatrix} r\\d \end{pmatrix} =
    \begin{pmatrix} br-ad \\ cd-er \end{pmatrix}
    \mbox{.} $$
    
    The required vanishing of Hom's will follow from
    $\mu F > \mu \Qc_x$, which is equivalent to
    $\frac{er-cd}{br-ad} < \frac{c}{a}$.
    From our assumptions, both denominators are positive, so this inequality is in turn equivalent to $a(er-cd) < c(br-ad)$. But this tautologically follows from $1 = \det \psi = bc-ae$, and this concludes the proof.
\end{proof}

Now we would like state our main theorem of this section, and for that we will first summarize the notation.

\begin{Not}[For Theorem \ref{Thrm:gen_birat}]
\label{Not:gen_birat}
    Recall that $X$ was the fixed elliptic surface with a fixed topological or oriented topological K-theory vector $v = (r, \Lambda, k) \in \Kt X$ such that 
    all semistable sheaves of class $v$ are stable. Fix an $\SLZ$ matrix such that $a>0$ and $e$ is a multiple of $\lambda_X$:
    $$  \varphi = \begin{pmatrix} c&a \\ e&b \end{pmatrix}
    \mbox{.}$$
    Define the elliptic surface $Y = J_X(a,b)$ and choose a Fourier-Mukai transform $\Phi = \FMYX^{\Pc}$ that corresponds to the matrix $\varphi$ as follows:
    $$ \mbox{For any } F \in \Db Y:
    \mbox{ }
    \begin{pmatrix} \rk (\Phi F) \\ \fdeg (\Phi F) \end{pmatrix} = \varphi 
    \begin{pmatrix} \rk F \\ \fdeg F \end{pmatrix}
    \mbox{.} $$
    Let $\Psi = \FMXY^\Qc$ be the ``almost inverse'' transform to $\Phi$ with the corresponding matrix $\psi \in \SLZ$, that is $\Qc = \Pc^\vee[1]$ as well as $\Phi \Psi = [1]$, $\Psi \Phi = [1]$, and $\varphi \psi = \psi \varphi$ are equal to the negative identity matrix. Moreover, we pick $\Pc$ in such a way that $\Pc_x$ and $\Qc_x$ are stable fiber sheaves on $Y$, for any $x \in X$.
    
    Let $\Mc_v(X)$ and $\Mc_w(Y)$ be the moduli spaces of stable sheaves of classes $v$ and $w$ on the elliptic surfaces $X$ and $Y$, respectively. 
\end{Not}

\begin{Thrm}
\label{Thrm:gen_birat}
    Use Notation \ref{Not:gen_birat}.
    Assume that $r>0$ and $d = \Lambda.f$ are coprime.
    Let $w$ be such that $\ch w = -(\ch \Psi)(\ch v)$.
    \begin{enumerate}[label=(\roman*)]
    \item 
    Assume that either $\rk w > 1$, or $\rk w = 1$ and $r>a$.
    Then $\Mc = \Mc_{v}(X)$ is birationally isomorphic to $\Nc = \Mc_w(Y)$.
    \item \label{FM_codim_two}
    If $\rk w = 1$ and $r > at$, for $t = \dim \Mc_{w\vt} - \dim \Pic X$, then the birational isomorphism is a regular isomorphism.
    \end{enumerate}
\end{Thrm}

\begin{proof}
    By Remark $\ref{Rem:numerology}$,
    our assumption $\rk w > 0$ is equivalent to $br-ad>0$, which is an assumption that is used in some of the previous lemmas, so we can apply those freely.

    By the results of the previous lemmas, there are two isomorphic open subsets, namely $\Mc^\circ \subset \Mc$ and $\Nc^\circ \subset \Nc$. By Bridgeland's Theorem \ref{Br_smooth}, the spaces $\Mc$ and $\Nc$ are irreducible, hence it is enough to prove that either of the open subsets are nonempty, and then the birationality will follow.
    
    Let us prove that $\Nc^\circ$ is nonempty.
    By Bridgeland's Lemma 6.5 \cite{Br}, a sheaf $F \in \Nc$ is $\Phi$-WIT$_0$ if and only if for any $x \in X$:
    $$\Hom{Y}{F,\Qc_x} = 0.$$
    But $\Qc_x$ is the pushforward of some stable torsion-free sheaf of slope $-\frac{c}{a}$. So if we prove that the slope of the torsion-free part of the restriction of $F$ to the fiber has greater slope, then we get the desired Hom-vanishing.

    For $\rk w > 1$, use Lemma \ref{Lem:vb_in_N} to observe that all vector bundles are in $\Nc^\circ$, and since the rank is greater than one, stable vector bundles exist.

    Assume now that $\rk w = 1$.
    Take $F \in \Nc$ which is a twist of an ideal sheaf corresponding to some finite length scheme $Z \subset Y$.
    For a positive number $s$, assume that scheme-theoretic intersection of $Z$ with any fiber has length at most $s$ -- this defines an open subset in $\Nc$.
    We will prove that under the assumption that $r > sa$ for some positive number $s$, the slope of the torsion-free part of $F_{\pi(x)}$ is greater than $-\frac{c}{a}$. Let us denote this torsion-free sheaf on a (possibly singular) fiber by $F'$.
    
    Note that from the matrix $\varphi$ we can read off  the rank and fiber degree of $F$ from the vector
    $\varphi^{-1} v = w$,
    so we get $1 = \rk F = br-ad$
    and $\fdeg F = cd - er$.
    Further, the intersection of $Z$ with the fiber decreases the slope of $F'$ by at most $s$, so we have $\mu F' \geq -s + \fdeg F > - \frac{r}{a} + cd-er$. Since $1 = \det \varphi = bc - ae$, we can rewrite the right hand side as follows:
    $$
    - \frac{r}{a} + cd-er =
    - (bc-ae) \frac{r}{a} + cd-er =
    - \frac{bcr}{a} + cd =
    - \frac{c}{a} (br - ad)
    .$$
    Further, recall that $1 = \rk w = br-ad$, so we in fact have:
    $$
    - \frac{r}{a} + cd-er =
    - \frac{c}{a}
    .$$
    Therefore, $\mu F' > - \frac{c}{a}$, and considering in particular $s=1$ gives the first part.
    For the second part, put $s=t$ and observe that this gives no restriction on the choice of an ideal, therefore in this case, we have $\Nc^\circ = \Nc$. And since $\Nc$ is proper, the birational isomorphism given in the first part should be an isomorphism of schemes.
\end{proof}

\begin{Cor}
\label{Cor:K3_birat}
    Use Notation \ref{Not:gen_birat}. Assume that $X$ is a complex elliptic K3 surface (and therefore $Y$ is, too). Assume also that $\rk w \geq 3$.
    Then $\Mc = \Mc_{v}(X)$ is birationally isomorphic to $\Nc = \Mc_w(Y)$, and the singular locus of the birational isomorphism has codimension at least two.
\end{Cor}

\begin{proof}
    It follows from the proof of Theorem \ref{Thrm:gen_birat} that with the additional assumption that $\rk w$ is at least 3, we can show birationality in codimension two for K3 surfaces. The modification is as follows.
    If $\rk w > 1$, then the locus of vector bundles is nonempty, hence by Lemma \ref{Lem:vb_in_N} we have that $\Nc^\circ$ is nonempty.
    Moreover, for $\rk w > 2$, the non-vector bundle locus is of codimension at least two, and hence the codimension of the complement $\Nc \setminus \Nc^\circ$ is at least two.
    Now we have a birational isomorphism of two hyperK\"ahler varieties, and the singular locus of the rational morphism $\Nc \to \Mc$ is of codimension at least two. Therefore, the singular locus of the inverse rational morphism $\Mc \to \Nc$ also is of codimension at least two (see \cite{MO}, end of page 2076, for the argument).
\end{proof}

\begin{Rem}[Difference from Bridgeland's results]

The reader may observe that even if we obtained birationalities between moduli spaces of sheaves of higher rank, we could still use Bridgeland's results as follows.
Taking two moduli spaces, we can find a Hilbert scheme of points to which both of them are birationally equivalent. So it is a reasonable question to ask why we did all the detailed work in this section.
In our argument, we say that these birational equivalences are given by a certain Fourier-Mukai transformation, and while it is natural to expect that a composition of Bridgeland's birationalities is given by a Fourier-Mukai kernel, too, we don't have a proof of this at hand.
We also see additional benefit of our formulation in that it allows for a more general view on the moduli spaces, and we have proved birationalities in the new setting of oriented K-theory vectors.
\end{Rem}
\section{Application: the Strange Duality for elliptic surfaces}

\label{Sec:SD}

In this section, we will prove the Strange Duality for certain K-theory vectors on elliptic surfaces. We will employ the Marian-Oprea trick of reducing the question about higher rank moduli spaces to Hilbert schemes, and use their result for Hilbert schemes.


Let $Y$ be an elliptic surface and let $\pi': Y \to C$ be the fibration with the class of a fiber $f$. Let $\lambda = \lambda_Y$ be the smallest possible positive fiber degree of a sheaf.

We assume that we work with two oriented topological K-theory vectors $v , w \in \Kor Y$ which are orthogonal:
\begin{equation*}
\begin{split}
    v &= (1 , \chn_1 v , \chn_2 v) = (1, \Lambda_v, k_v);\\
    w &= (1 , \chn_1 w , \chn_2 w) = (1, \Lambda_w, k_w);\\
    \Oc(\Lambda_v + \Lambda_w) & \mbox{ does not have higher cohomology.}
\end{split}
\end{equation*}

So from now on, when we write the vectors $v$ and their associated geometric objects $\Mc_v$, we assume that the determinant is fixed. We do it because we use the following theorem 
\ref{Thrm:MO_Hilb} by Marian-Oprea as our base case.
They consider Hilbert schemes of points -- in this case, the determinant is fixed -- and for these,
we have the Strange Duality isomorphism.

\begin{Thrm}[Marian-Oprea \cite{MO08}, Proposition 1]
\label{Thrm:MO_Hilb}
For any surface $X$, for any line bundle $L$ on $X$ with $\chi(L) = n$ and no higher cohomology,
the theta locus is a divisor and the Strange Duality map is an isomorphism for a pair of vectors of the form
$v = (1,\Oc_X,-k) \in \Kor X$ and
$w = (1,L,\frac{1}{2} L^2 - l) \in \Kor X$, where $k$ and $l$ are positive integers that sum up to $k+l=n$.
\end{Thrm}

Let $\Mc_v$ and $\Mc_w$ denote the two Hilbert schemes of points on $Y$ (twisted by the line bundles $\Lambda_v$ and $\Lambda_w$) corresponding to $v$ and $w$, respectively.

Pick a matrix $\varphi \in \SLZ$:
$$ \varphi = \begin{pmatrix} c&a \\ e&b
\end{pmatrix}
$$
with $a>0$, the entry $e$ being the multiple of $\lambda$, and both $c$ and $-b$ greater than $a$.

Denote by $X$ the moduli space of stable fiber sheaves $Y = J_X(a,c)$ of rank $a$ and degree $c$, and let $\pi: X \to C$ be the fibration morphism, as $X$ is also an elliptically fibered surface with $\lambda_X = \lambda$.

There are tautological sheaves $\Pc$ and $\Qc$ on $X \times Y$ such that $\Qc$ para\-met\-rizes stable fiber sheaves on $Y$ of rank $a$ and degree $c$, and in addition $\Pc = \Qc^\vee[1]$ is a sheaf.
Since stability of a torsion-free sheaf is equivalent to stability of its nonderived dual, we can observe that the restriction $\Pc_x$ of $\Pc$ to any fiber of $X \times Y \to X$ also represents a stable sheaf on an elliptic fiber in $Y$.

\paragraph{Fourier-Mukai transforms.} For a recollection of the main definitions, see \S\ref{Subsec:FM_stage}. Let $\Phi$ denote the Fourier-Mukai transform $\FMYX^\Pc$ from $\Db Y$ to $\Db X$ with kernel $\Pc$. Let $\Qc \deq (\Pc^\vee \otimes \pi^* \omega_X)[1]$ and $\Xi \deq \FMYX^\Qc$. By Bridgeland's Theorem \ref{Thrm:Br_birat}, the functors $\Phi$ and $\Xi$ define equivalences of derived categories, and their quasi-inverses, up to a shift, are $\Psi \deq \FMXY^\Qc$ and $\Omega \deq \FMXY^\Pc$.
Pictorially, we represent the functors and their corresponding matrices as follows:
\begin{diagram}
    & \mbox{Kernel } \Pc &&&& \mbox{Kernel } \Qc & \\
    &\omega = \begin{pmatrix}b&a\\e&c\end{pmatrix} &&&&
    \psi = \begin{pmatrix}-b&a\\e&-c\end{pmatrix} \\
    \Db X & \pile{\rTo^\Omega \\ \lTo_\Phi} & \Db Y
    &&
    \Db X & \pile{\rTo^\Psi \\ \lTo_\Xi} & \Db Y \\
    &\varphi = \begin{pmatrix}c&a\\e&b\end{pmatrix} &&&&
    \xi = \begin{pmatrix}-c&a\\e&-b\end{pmatrix} \\
\end{diagram}
The matrices here represent the action of the functors on the vector consisting of rank and fiber degree of a sheaf, for example:
$$  \begin{pmatrix} \rk \Phi E \\ \fdeg \Phi E \end{pmatrix} =
    \varphi \begin{pmatrix} \rk E \\ \fdeg E \end{pmatrix}
.$$

We would like to list the relations these functors satisfy, obtained from Bridgeland's setup and the Grothendieck-Riemann-Roch theorem, for any $E \in \Db X$:
\begin{equation}
\label{Eq:functor_relations}
\begin{split}
    \Phi\Psi &= \Xi \Omega = [-1] , \\
    \Psi\Phi &= \Omega \Xi = [-1] , \\
    \Psi (E^\vee) & = (\Omega E)^\vee [-1]
\mbox{.}
\end{split}
\end{equation}

\paragraph{Main objective.} We want to prove that for $E \in \Mc_v$ and $E' \in \Mc_w$ away from a codimension two locus, the definition of the theta-divisor on $\Mc_v(Y) \times \Mc_w(Y)$ -- that is, the locus where $\Hrm^0(E \otimes E') \neq 0$ -- can be transported to a definition of the theta-divisor on
$$\Mc_{\overline{\Xi v}}(X) \times \Mc_{\overline{\Phi w} \otimes K_X}(X)$$ using the Fourier-Mukai equivalences. Overlines denote taking derived dual of a representative of the K-theory class. 
Namely, we consider:
\begin{equation*}
\begin{split}
    \Hrm^1(Y, E \otimes E')
    &= \Hrm^1 \, \RHom{Y}{E^\vee , E'}
    = \Hrm^1 \, \RHom{X}{\Phi (E^\vee) , \Phi E'}
    = \\ &
    = \Hrm^1 \, \RHom{X}{(\Xi E)^\vee [-1] , \Phi E'}
    = \Hrm^2 (X, \Xi E \otimes \Phi E')
    = \\ &
    = \Hrm^0 \big(X, (\Xi E)^\vee \otimes (\Phi E')^\vee \big)^\vee
\mbox{,}
\end{split}
\end{equation*}
where we start by applying the functor $\Psi$, relations \eqref{Eq:functor_relations}, and finally the Serre duality.

So if we can prove that $\Xi$ and $\Phi$ establish birational isomorphisms in codimension two, then we will be able to deduce that the theta locus on $\Mc_{\overline{\Xi v}}(X) \times \Mc_{\overline{\Phi w}}(X)$
is a divisor and the theta line bundles have the same sections as their counterparts on the side of Hilbert schemes $\Mc_v \times \Mc_w$.

We can now state one of our main results, which is a direct consequence of Corollary \ref{Cor:K3_birat}:

\begin{Thrm}
\label{Thrm:SD_elliptic_K3}
    Keep the notation of this section.
    If we assume that
    $d_v \deq \Lambda_v.f > 2 + \frac{c}{a}$ and
    $d_w \deq \Lambda_w.f > 2 - \frac{c}{a}$,
    then the theta locus is a divisor and the Strange Duality holds for the pair of vectors $\overline{\Xi v}$ and $\overline{\Phi w}$.
\end{Thrm}

\begin{proof}
    To satisfy the assumptions of Corollary \ref{Cor:K3_birat}, we need to make sure that $\rk \Xi v \geq 3$ and $\rk \Phi w \geq 3$. From the matrix representation above, we can find a formula for $\rk \Xi v$:
    $$  \rk \Xi v = \left( \xi \cdot
    \begin{pmatrix} \rk v \\ \fdeg v \end{pmatrix} \right)^1 =
    \begin{pmatrix} -c & a \end{pmatrix}
    \begin{pmatrix} 1 \\ d_v \end{pmatrix}
    = ad_v - c
    \mbox{.}
    $$
    The requirement that $\rk \Xi v > 2$ is now evidently equivalent to asking that $d_v > 2 + \frac{c}{a}$. Rewriting $\rk \Phi w \geq 3$ in terms of $\ch w$ is fully analogous. 
\end{proof}

By returning the twist by the canonical bundle to the equations, we can obtain a result for general elliptic surfaces; since the moduli spaces of sheaves on those are not necessarily hyperK\"ahler, we need to assume stronger bounds:

\begin{Thrm}
\label{Thrm:SD_elliptic}
    Keep the notation of this section and let
    $t_v = \dim \Mc_{v\vor}$.
    If we assume that
    $d_v > t_v + \frac{c}{a}$ and
    $d_w > t_w - \frac{c}{a}$,
    then the theta locus is a divisor and the Strange Duality holds for the pair of vectors $\overline{\Omega v}$ and $\overline{\Psi w} \otimes K_X$.
\end{Thrm}


\section{New universal sheaves as Fourier-Mukai kernels}

\label{Sec:matrices}

In this section, we will present universal sheaves $\Pc_d, \Fc_d$ on the fiber square of an elliptic K3 surface $X$ with a section. The construction mimicks Atiyah's classification of vector bundles on elliptic curves. Further, we explicitly describe the action of the corresponding Fourier-Mukai functors $\Phi$ (with kernel $\Pc_d$ or $\Fc_d$) on the Chow groups:
$\ch \Phi : \Chow X \to \Chow X$. The author thanks Tony Pantev for suggesting the construction of $\Pc_d$ and $\Fc_d$ in private correspondence \cite{Tony}.

\subsection{Construction of new universal sheaves}
\label{Subsec:matrices:univ_shvs}

Let $\pi : X \to \PS^1$ be an elliptic K3 surface with a section $\sigma : \PS^1 \to X$ (that does not pass through singular points of fibers). The section class in $\Chow X$ is also denoted by $\sigma$, and let $f$ denote the class of a fiber.

Let $\Pi \deq X \times_{\PS^1} X$ -- it is a divisor in $X \times X$, and let $\iota : \Pi \to X \times X$ denote the embedding.
Let $\Delta \cong X$ be the diagonal, and let $\kappa : \Delta \to \Pi$ denote the diagonal embedding into the fibered product.

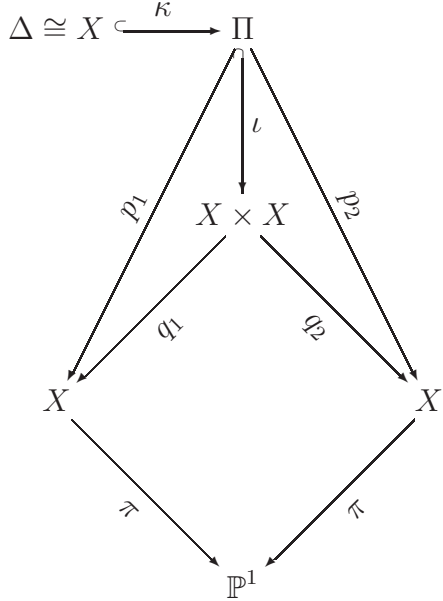
\begin{diagram}
\Delta \cong X & \rInto^{\,\,\,\,\,\kappa} & \Pi           &\\
&\ldTo(2,4)^{p_1}& \dInto_\iota & \rdTo(2,4)^{p_2} \\
&& X \times X    &\\
& \ldTo_{q_1} &  & \rdTo_{q_2}       \\
X &&  && X \\
& \rdTo_{\pi} && \ldTo_{\pi}       \\
&& \PS^1    &\\
\end{diagram}

Denote the natural projection $\Pi \to \PS^1$ by $\rho$:
$$ \rho = \pi q_1 \iota = 
\pi q_2 \iota =
\pi p_1 = \pi p_2:
\Pi \to \PS^1
.$$

Define the universal sheaf $\Pc_d$ on $\Pi$ classifying degree $d$ fiber sheaves on $X$, for a positive $d$:
$$
\Pc_d = \Oc_\Pi \Big(
    -\Delta + (d+1)p_1^*(\sigma)
    + p_2^* (\sigma)
    \Big) \otimes 
    \rho^* \Oc_{\PS^1}(2(d+1))
. $$

\begin{Lem}[\cite{Tony}]
    The sheaf $\Pc_d$ satisfies the following properties:
    \begin{enumerate}[label = (\roman*)]
        \item 
        If $t \in \PS^1$ and $x \in X_t = \pi^{-1}(t)$ a smooth point of the fiber over $t$, then
        $\Pc_{d |X_t \times x} \cong 
        \Oc_{X_t} ((d+1)\sigma(t) - x)$.
        \item 
        $\Pc_d$ is normalized, that is
        $\Pc_{d |\sigma \fprod{\PS^1} X} \cong \Oc_X$.
    \end{enumerate}
\end{Lem}

\begin{proof}
Note that we can write 
\begin{equation*}
\begin{split}
\Pc_d &\cong \Oc_\Pi(-\Delta) \otimes \\
    &\otimes \iota^* \Oc((d+1) [\sigma \times X]) \\
    &\otimes \iota^* \Oc([X \times \sigma]) \\
    &\otimes \rho^* \Oc_{\PS^1}(2(d+1)) ,\\
\end{split}
\end{equation*}
And since restriction commutes with tensor product, we can calculate the restriction of each factor separately. Note that we are working with line bundles, so classical restriction coincides with derived restriction.

\begin{itemize}
    \item $\Oc_\Pi(\Delta)_{|\sigma \fprod{\PS^1} X}
    = \Oc_\Delta (\sigma \fprod{\PS^1} \sigma) 
    = \Oc_X(\sigma)$.
    \item $\Oc(\sigma \times X)_{|\sigma \fprod{\PS^1} X} =
    \Oc_{X\times X} (q_1^* [\sigma]) _{|\sigma \fprod{\PS^1} X}$.
    The composition
    $$\sigma \fprod{\PS^1} X \to X \times X \xrightarrow{q_1} X$$
    coincides with the composition of the projection on the first factor
    $s_1: \sigma \fprod{\PS^1} X$ with the section:
    $$ \sigma \fprod{\PS^1} X \xrightarrow{s_1} 
    \sigma \xrightarrow{\sigma} X
    ,$$
    therefore we can write the line bundle that we are interested in as the pullback along the second composition:
    $$\Oc(\sigma \times X)_{|\sigma \fprod{\PS^1} X} =
    s_1^* \sigma^* \Oc_X(\sigma)
    .$$
    The sheaf $\sigma^* \Oc_X(\sigma)$ is the normal sheaf of $\sigma \subset X$, and both schemes are smooth, thus $\sigma^* \Oc_X(\sigma)$ can be expressed as the quotient:
    $$  0 \to \Tc_\sigma \to
    \Tc_{X | \sigma} \to
    \sigma^* \Oc_X(\sigma) \to 0
    .$$
    Recalling that $\sigma \cong \PS^1$ and $X$ is a K3 surface, we get
    $\sigma^* \Oc_X(\sigma) = \det \Tc_{X | \sigma} \otimes \Tc_\sigma^\vee
    \cong \Oc_{\PS^1}(-2)$. Therefore:
    $$  \big( \iota^* \Oc(\sigma \times X)
    \big)_{|\sigma \fprod{\PS^1} X} \cong
    s_1^* \Oc_{\PS^1} (-2) \cong \Oc_X(-2f)
    .$$
    \item
    $\Oc_{X\times X}([X \times \sigma])_{
    |\sigma \fprod{\PS^1} X} =
    \Oc_\Delta ([\sigma \fprod{\PS^1} \sigma]) = \Oc_X(\sigma)$.
    \item
    $\rho^* \Oc_{\PS^1}(1) =
    \Oc_X(f)$.
\end{itemize}

Putting this together, we get the second part of the lemma:
\begin{equation*}
\begin{split}
\Pc_{d| \sigma \fprod{\PS^1} X}
&\cong \Oc_X(-\sigma) \otimes
\Oc_X((d+1) \cdot (-2f))\\
&\otimes \Oc_X(\sigma)
\otimes \Oc_X(2(d+1)f)
\cong \Oc_X.
\end{split}
\end{equation*}

\end{proof}

We now want to construct the universal sheaf $\Fc_d$ classifying rank $d+1$, degree $d$ stable fiber sheaves as a universal extension:
\begin{equation*}
    0 \to p_2^* \left(
    p_{2*} \Pc_d \otimes \omega_\pi \right)
    \to \Fc_d \to \Pc_d \to 0
\end{equation*}

For this, we will need several lemmas, and we start with observing that $p_{2*} \Pc_d$ is a sheaf and establishing
that there is a canonical extension class corresponding to $\Fc_d$. Recall that in our definition of the universal sheaf $\Pc_d$, we assumed that $d>0$.

\begin{Lem}[\cite{Tony}]
The derived pushforward $p_{2*} \Pc_d$ is a vector bundle of rank $d$ on $X$ (in particular, a complex concentrated in degree zero).
\end{Lem}

\begin{proof}
By cohomology and base change, restriction of $p_{2*} \Pc_d$ to a point $x \in X$ is isomorphic to the cohomology of the restriction $\Pc_{d|X\fprod{\PS^1}x}$.
The latter is a line bundle of positive degree $d$ on an elliptic curve (possibly singular), hence its cohomology is concentrated in degree zero and has dimension $$\dim \Hrm^0 \left(X\fprod{\PS^1}x, \Pc_{d|X\fprod{\PS^1}x} \right) 
= \deg \Pc_{d|X\fprod{\PS^1}x} = d.$$
So $p_{2*} \Pc_d$ is a sheaf of constant rank $d$.
\end{proof}

\begin{Lem}[\cite{Tony}]
We have a canonical isomorphism:
$$\Ext^1_\Pi (\Pc_d , p_2^* \left(
p_{2*} \Pc_d \otimes \omega_\pi \right)) =
\Hom{X}{p_{2*}\Pc_d, p_{2*}\Pc_d}
.$$
\end{Lem}

\begin{proof}
If $\omega_{p_2}$ and $\omega_\pi$ denote the relative dualizing sheaves for morphisms $p_2$ and $\pi$, respectively, then we note that $\omega_{p_2} \cong p_2^* \omega_\pi$, and then by relative duality and projection formula we have:
$$ p_{2*} (\Pc_d^\vee) = (p_{2*} (\Pc_d \otimes \omega_{p_2}))^\vee \cong
(p_{2*}\Pc_d \otimes \omega_\pi)^\vee
.$$
We can apply this to get the following sequence of natural isomorphisms:
\begin{equation*}
\begin{split}
\RHom{\Pi}{\Pc_d , p_2^* (p_{2*} \Pc_d \otimes \omega_\pi)} &=
\RGamma_\Pi (\Pc_d^\vee \otimes
p_2^* ( p_{2*} \Pc_d \otimes \omega_\pi ) ) =\\
&= \RGamma_X \circ p_{2*}
( \Pc_d^\vee \otimes
p_2^* ( p_{2*} \Pc_d \otimes \omega_\pi ) ) =\\
&= \RGamma_X (
p_{2*} (\Pc_d^\vee) \otimes
(p_{2*} \Pc_d \otimes \omega_\pi) ) =\\
&= \RGamma_X (
(p_{2*} \Pc_d \otimes \omega_\pi)^\vee
\otimes
(p_{2*} \Pc_d \otimes \omega_\pi) ) =\\
&= \RHom{X}{p_{2*} \Pc_d \otimes \omega_\pi, p_{2*} \Pc_d \otimes \omega_\pi}
.
\end{split}
\end{equation*}
\end{proof}

\subsection{Action on Chow ring}
\label{Subsec:matrices:action}

Now we want to compute the action of the Fourier-Mukai functor with kernel $\Fc_d$ on cohomology. For that, we start with computing the action of another functor with kernel $\Pc_d$.

We use the notation from the previous section.

First note that for any $E \in \Db X$, we have a canonical isomorphism:
$$p_{1*} \left( \Pc_d \otimes p_2^* E
\right) =
q_{1*} \left( \iota_*\Pc_d \otimes q_2^* E
\right)
.$$

Now, by projection formula:
\begin{equation*}
\begin{split}
& \iota_* \Pc_d = (\iota_* \Oc_\Pi(-\Delta))
    \otimes q_1^* \Lc_1 \otimes q_2^* \Lc_2
    \mbox{, where:}\\
& \Lc_1 = \Oc_X( (d+1) \sigma + 2(d+1)f ),\\
& \Lc_2 = \Oc_X(\sigma).
\end{split}    
\end{equation*}

\begin{Lem}
Let $E \in \Db X$. Then the cohomology class 
$\ch (\FM^{\iota_* \Pc_d} (E))$ is equal to:
$$ \ch \left[ \FM^{\iota_* \Pc_d} (E) \right] =
\ch [\Lc_1] \otimes
\ch \Big(
\pi^* \pi_* [E(\sigma)] - [E(\sigma)]
\Big)
.$$
\end{Lem}

\begin{Lem}
Let $E \in \Db X$. 
The matrix of $[\Lc_1] \otimes -$:
$$ \begin{bmatrix}
 1&0&0 \\
 (d+1)(\hat \sigma + 2 \hat f) & 1 & 0 \\
 (d+1)^2 & (d+1)(\hat \sigma + 2 \hat f) & 1
\end{bmatrix}
$$
$$ \begin{bmatrix}
 1 & 0 & 0 & 0 \\
 d+1 & 1 & 0 & 0 \\
 2(d+1) & 0 & 1 & 0 \\
 (d+1)^2 & 0 & d+1 & 1
\end{bmatrix}
$$
\end{Lem}

\begin{Lem}
Let $E \in \Db X$. 
The matrix of $[\Oc_X(\sigma)] \otimes -$:
$$ \begin{bmatrix}
 1&0&0 \\
 \hat \sigma & 1 & 0 \\
 -1 & \hat \sigma & 1
\end{bmatrix}
$$
$$ \begin{bmatrix}
 1 & 0 & 0 & 0 \\
 1 & 1 & 0 & 0 \\
 0 & 0 & 1 & 0 \\
 -1& -2& 1 & 1
\end{bmatrix}
$$
\end{Lem}

\begin{Lem}
Let $E \in \Db X$. 
The matrix of $[\pi^* \pi_* (-)]$:
$$ \begin{bmatrix}
 0      & \fdeg         & 0 \\
 2\hat f& -\hat f \circ \fdeg
        & \hat f \circ \deg \\
0       & 0 & 0
\end{bmatrix}
$$
$$ \begin{bmatrix}
 0 & 1 & 0 & 0 \\
 0 & 0 & 0 & 0 \\
 2 & -1& 0 & 1 \\
 0 & 0 & 0 & 0 \\
\end{bmatrix}
$$
\end{Lem}

\begin{Lem}
Let $E \in \Db X$. 
The matrix of $[\pi^* \pi_* (-(\sigma))]$:
$$ \begin{bmatrix}
 1  & \fdeg         & 0 \\
 0  & \hat f \circ \deg \circ (\hat \sigma - \hat f)
    & \hat f \circ \deg \\
0   & 0 & 0
\end{bmatrix}
$$
$$ \begin{bmatrix}
 1 & 1 & 0 & 0 \\
 0 & 0 & 0 & 0 \\
 0 & -3& 1 & 1 \\
 0 & 0 & 0 & 0 \\
\end{bmatrix}
$$
\end{Lem}

\begin{Prop}
The matrix of $\FM^{\iota_* \Pc_d} (E)$ is:
$$ \begin{bmatrix}
0&\fdeg&0 \\
-\hat \sigma &
-1 + (d+1)\hat \sigma \circ \fdeg 
+ \hat f \circ \deg ( (2d+1) \hat f + \hat \sigma) &
\hat f \circ \deg \\
1 & (d+1)(d-2)\hat f - \hat \sigma & d
\end{bmatrix}
$$
$$ \begin{bmatrix}
 0 & 1      & 0 & 0 \\
 -1& d      & 0 & 0 \\
 0 & 2d-1   & 0 & 1 \\
 1 & d^2-d  & -1 & d \\
\end{bmatrix}
$$
\end{Prop}

\begin{Lem}
Let $E \in \Db X$. Then the Chern character of
$\FM^{\iota_* \Fc_d} (E)$ is equal to:
$$ \ch \left[ \FM^{\iota_* \Fc_d} (E) \right] =
\ch \left[ \FM^{\iota_* \Pc_d} (E) \right] +
\ch \Big(
\pi^* \pi_*
[
E \otimes p_{2*} \Pc_d \otimes \omega_\pi
] \Big)
.$$
\end{Lem}

\begin{Prop}
The matrix of $[p_{2*} \Pc_d \otimes \omega_\pi] \otimes -$ is:
$$ \begin{bmatrix}
d &0&0 \\
-\hat \sigma+(d^2+d)\hat f &
d & 0 \\
-2d-1 & 
-\hat \sigma+(d^2+d)\hat f & d
\end{bmatrix}
$$
$$ \begin{bmatrix}
 d      & 0          & 0 & 0 \\
 -1     & d          & 0 & 0 \\
 d^2+d  & 0          & d & 0 \\
 -2d-1  & d^2+d+2    & -1 & d \\
\end{bmatrix}
$$
\end{Prop}

\begin{Prop}
The matrix of $\FM^{\iota_* \Fc_d} (E)$ is:
$$ \begin{bmatrix}
-1  & (d+1) \fdeg   & 0 \\
-\hat \sigma &
-1 + \Big(
    (d+1)\hat \sigma + (d+1)^2 \hat f
    \Big) \circ \fdeg &
(d+1) \hat f \circ \deg \\
1 &  - \hat \sigma + (d^2-d-2)\hat f & d
\end{bmatrix}
$$
$$ \begin{bmatrix}
 -1& d+1    & 0 & 0 \\
 -1& d      & 0 & 0 \\
 0 & (d+1)^2& -1 & d+1 \\
 1 & d^2-d  & -1 & d \\
\end{bmatrix}
$$
\end{Prop}

\printbibliography[
    heading=bibintoc,
    title={Bibliography}
    ]

\end{document}